\documentclass[12pt,reqno,final]{amsart}

\usepackage{amsmath}
\usepackage{amsthm}
\usepackage{amsfonts}  %mathbb, mathbf, mathfrac
\usepackage{amssymb}   %mathbb, mathbf
\usepackage{latexsym}

\theoremstyle{plain}
\newtheorem{theorem}{Theorem}
\newtheorem{proposition}[theorem]{Proposition}
\newtheorem{corollary}[theorem]{Corollary}
\newtheorem{lemma}[theorem]{Lemma}

\theoremstyle{remark}
\newtheorem{remark}{Remark}
\newtheorem{example}{Example}

\DeclareMathAlphabet\mathoo{U}{eur}{b}{n}
\DeclareMathOperator{\Real}{Re}
\DeclareMathOperator{\co}{ch}%{conv\;hull}

\usepackage[hypertex,colorlinks=true,citecolor=blue]{hyperref} %Для YAP

\begin{document}%\large
\title[The Gelfand--Shilov type estimate]{The Gelfand--Shilov type estimate\\ for Green's function\\ of the bounded solutions problem}

\author{V.G. Kurbatov}
 \address{Department of Mathematical Physics,
Voronezh State University\\ 1, Universitetskaya Square, Voronezh 394018, Russia}
 \thanks{$^*$ The first author was supported by the Ministry of Education and Science of the Russian Federation under state order No.~3.1761.2017.}
 \email{kv51@inbox.ru}

\author{I.V. Kurbatova}
 \address{Department of Software Development and Information Systems Administration,
Vo\-ro\-nezh State University\\ 1, Universitetskaya Square, Vo\-ro\-nezh 394018, Russia}
 \thanks{$^*$ The second author was supported by the Russian Foundation for Basic Research under research project No.~16-01-00197.}
 \email{la\_soleil@bk.ru}

 %Functional calculus


\begin{abstract}
An analogue of the Gelfand--Shilov estimate of the matrix exponential is proved for Green's function of the problem of bounded solutions of the ordinary differential equation $x'(t)-Ax(t)=f(t)$.
\end{abstract}

\maketitle

\section*{Introduction}\label{s:Introduction}
In~\cite[p.~68, formula (13)]{Gelfand-Shilov-GF3:eng}, it was established the following statement. Let the eigenvalues of an $N\times N$-matrix $A$ lie in the half-plane $\Real\lambda<-\gamma$. Then the matrix exponential satisfies the estimate
\begin{equation*}
\Vert e^{At}\Vert\le e^{-\gamma t}\sum_{j=0}^{N-1}c_jt^j,\qquad t>0,
\end{equation*}
where the coefficients $c_j\ge0$ depend only on $\Vert A\Vert$ (see Corollary~\ref{p:Bylov-Vinograd} for details).
In particular, it easily follows from this estimate that
$\lim_{t\to+\infty}e^{(\gamma-\varepsilon)t}\Vert e^{At}\Vert=0$ for any $\varepsilon>0$ uniformly for any bounded family of matrices $A$.
Applications of estimates of $\Vert e^{At}\Vert$ can be found in~\cite{Bylov-Vinograd:rus1,Gelfand-Shilov-GF3:eng,Gil1993Estimates}.

In this paper, we prove a similar estimate for Green's function for the problem of bounded on the axis solutions of the differential equation
\begin{equation*}
x'(t)-Ax(t)=f(t).
\end{equation*}
The proof is similar to that of~\cite{Gelfand-Shilov-GF3:eng} and uses some constructions from~\cite{Kurbatov-Kurbatova17:Arxiv}.

In Sections~\ref{s:Newton polynomial} and~\ref{s:Matrix functions}, preliminaries are collected. In Section~\ref{s:Green}, we recall the definition of Green's function and some its properties and describe its representation in the form of the Newton interpolating polynomial. In Section~\ref{s:estimate} we prove our estimate (Theorem~\ref{t:VinGr}).

\section{The Newton interpolating polynomial}\label{s:Newton polynomial}

Let $\mu_1$, $\mu_2$, \dots, $\mu_N$ be given complex numbers (some of them may coincide
with others) called \emph{points of interpolation}. Let a complex-valued function $f$ be
defined and analytic in a neighbourhood $U$ of these points. \emph{Divided differences}
of the function $f$ with respect to the points $\mu_1$,
$\mu_2$, \dots, $\mu_N$ are defined (see, e.g.,~\cite{Gelfond:eng,Jordan}) by the recurrent
relations
 \begin{equation*}%\label{e:divided differences}
 \begin{split}
f[\mu_i]&=f(\mu_i),\\
f[\mu_i,\mu_{i+1}]&=\frac{f(\mu_{i+1})-f(\mu_i)}{\mu_{i+1}-\mu_i},\\
f[\mu_i,\dots,\mu_{i+m}]&=\frac{f(\mu_{i+1},\dots,\mu_{i+m})
-f(\mu_{i},\dots,\mu_{i+m-1})} {\mu_{i+m}-\mu_i}.
 \end{split}
 \end{equation*}
In these formulas, if the denominator vanishes, then the quotient is understood as the derivative with respect to the corresponding argument of the previous divided difference. %(this agreement can by derived by continuity from Corollary~\ref{c:f[] is continuous}).

\begin{proposition}[{\rm\cite[ch.~1, formula (54)]{Gelfond:eng}}]\label{p:f[] via Gamma}
Let the function $f$ be analytic in a neighbourhood of the points of interpolation $\mu_1$, $\mu_2$, \dots, $\mu_N$. Then
\begin{equation*}
f[\mu_{1},\dots,\mu_{N}]=\frac1{2\pi i}\int_{\Gamma}\frac{f(z)}{\Omega(z)}\,dz,
\end{equation*}
where the contour $\Gamma$ encloses all the points of interpolation and
\begin{equation*}
\Omega(z)=\prod_{k=1}^N(z-\mu_k).
\end{equation*}
\end{proposition}

%\begin{corollary}\label{c:f[] is continuous}
%Divided differences are differentiable functions of their arguments.
%\end{corollary}
%\begin{proof}
%The statement follows from Proposition~\ref{p:f[] via Gamma}.
%\end{proof}

% \begin{corollary}\label{c:f[] is indep of order}
%Divided differences $f[\mu_{1},\dots,\mu_{N}]$ are
%symmetric function, i.e., they do not depend on the order of their arguments $\mu_1$, \dots, $\mu_{N}$.
% \end{corollary}
% \begin{proof}
%The statement follows from Proposition~\ref{p:f[] via Gamma}.
% \end{proof}

\begin{proposition}[{\rm\cite[ch.~1, formula (48)]{Gelfond:eng}}]\label{p:repr of Delta}
Let the points of interpolation $\mu_j$ be distinct. Then
\begin{equation*}
f[\mu_{1},\dots,\mu_{N}]=\sum_{j=1}^N\frac{f(\mu_j)}{\prod\limits_{\substack{k=1\\k\neq j}}^N(\mu_j-\mu_k)}.
\end{equation*}
\end{proposition}
 \begin{proof}
The statement follows from Proposition~\ref{p:f[] via Gamma}.
 \end{proof}

%\begin{proposition}[{\rm\cite[ch.~1, formula (47)]{Gelfond:eng}}]\label{p:Gelfond(47)}
%\begin{multline*}
%f[\mu_0,\mu_{1},\dots,\mu_{N}]=\int_0^1\int_0^{t_1}\dots\int_0^{t_{N-1}}f^{{(N)}}
%\bigl(\mu_0+(\mu_1-\mu_0)t_1+\dots\\
%\dots+(\mu_{N-1}-\mu_{N-2})t_{N-1}+(\mu_N-\mu_{N-1})t_N\bigr)\,dt_N\dots dt_1.
%\end{multline*}
%\end{proposition}

\begin{proposition}[{\rm\cite[ch.~1, formula (49)]{Gelfond:eng}}]\label{p:Gelfond(49)}
Let the domain $U$ of $f$ contain the convex hull of the set $\{\mu_{1},\dots,\mu_{N}\}$.
Then the following estimate holds{\rm:}
\begin{equation*}
\bigl|f[\mu_{1},\dots,\mu_{N}]\bigr|\le\frac1{(N-1)!}\max_{\lambda\in\co\;\{\mu_{1},\dots,\mu_{N}\}}|f^{(N-1)}(\lambda)|,
\end{equation*}
where $\co\;\{\mu_{1},\dots,\mu_{N}\}$ means the convex hull of the set $\{\mu_{1},\dots,\mu_{N}\}$.
\end{proposition}

The set $\lambda_1,\dots,\lambda_M\in\mathbb C$ of \emph{interpolation points} together with the set $n_1,\dots,n_M\in\mathbb N$ of their \emph{multiplicities} is called \emph{multiple interpolation data}.
We set $N=n_1+\dots+n_M$.

Let $U\subseteq\mathbb C$ be an open neighbourhood of the set $\lambda_1,\dots,\lambda_M$ of the points of interpolation and $f:\,U\to\mathbb C$ be an analytic function.
An \emph{interpolating polynomial} of $f$ that corresponds to the multiple interpolation data is a polynomial $p$ of degree степени $N-1$ satisfying the equalities
\begin{equation*}
p^{(j)}(\lambda_k)=g^{(j)}(\lambda_k),\qquad k=1,\dots,M;\;j=0,1,\dots,n_k-1.
\end{equation*}

\begin{proposition}[{\rm\cite[p.~20]{Jordan}}]\label{p:Newton poly}
For any analytic function $f$, the interpolating polynomial exists and unique.
Let $\mu_1,\dots,\mu_N$ be the points of multiple interpolation data $\lambda_1,\dots,\lambda_M$, listed in an arbitrary order and repeated as many times as their multiplicities $n_1,\dots,n_M$.
Then the interpolating polynomial possesses the representation
\begin{equation}\label{e:poly Newton}
\begin{split}
p(z)&=f[\mu_1]+f[\mu_1,\mu_2](z-\mu_1)+
f[\mu_1,\mu_2,\mu_3](z-\mu_1)(z-\mu_2)\\
&+f[\mu_1,\mu_2,\mu_3,\mu_4](z-\mu_1)(z-\mu_2)(z-\mu_3)+\dots\\
&+f[\mu_1,\mu_2,\dots,\mu_N](z-\mu_1)(z-\mu_2)\dots(z-\mu_{N-1}).
 \end{split}
\end{equation}
\end{proposition}

Representation~\eqref{e:poly Newton} is called~\cite{Gelfond:eng,Jordan} the \emph{interpolating polynomial in the Newton form} or shortly \emph{the Newton
interpolating polynomial} with respect to the points $\mu_1$, $\mu_2$, \dots, $\mu_N$.

\section{Matrix functions}\label{s:Matrix functions}

Let $A$ be a complex $N\times N$-matrix.
Let $\mathbf1$ be the identity matrix. The polynomial
\begin{equation*}
p_A(\lambda)=\det(\lambda\mathbf1-A)
\end{equation*}
is called the \emph{characteristic polynomial} of the matrix $A$. Let $\lambda_1$, \dots, $\lambda_M$ be the complete set of the roots of the characteristic polynomial $p_A$, and $n_1$, \dots, $n_M$ be their multiplicities; thus $n_1+\dots+n_M=N$. It is well known that $\lambda_1$, \dots, $\lambda_M$ are \emph{eigenvalues} of $A$. The numbers $n_k$ are called (\emph{algebraic{\rm)} multiplicities} of the eigenvalues $\lambda_k$. The set $\sigma(A)=\{\,\lambda_1, \dots, \lambda_M\,\}$ is called the \emph{spectrum} of $A$.

Let $U\subseteq\mathbb C$ be an open set that contains the spectrum $\sigma(A)$.
Let $f:\,U\to\mathbb C$ be an analytic function. The \emph{function $f$ of the matrix} $A$ is defined~\cite[p.~17]{Daletskii-Krein:eng},~\cite[ch.~VII]{Dunford-Schwartz-I:eng},~\cite[ch.~V, \S~1]{Hille-Phillips:eng} by the formula
\begin{equation*}
f(A)=\frac1{2\pi i}\int_\Gamma f(\lambda)(\lambda\mathbf1-A)^{-1}\,d\lambda,
\end{equation*}
where the contour $\Gamma$ surrounds the spectrum $\sigma(A)$.

\begin{proposition}[{\rm\cite[Theorem 5.2.5]{Hille-Phillips:eng}}]\label{p:func calc}
The mapping $f\mapsto f(A)$ preserves algebraic operations, i.~e.,
\begin{align*}
(f+g)(A)&=f(A)+g(A),\\
(\alpha f)(A)&=\alpha f(A),\\
(fg)(A)&=f(A)g(A),
\end{align*}
where $f+g$, $\alpha f$ and $fg$ are defined pointwise.
\end{proposition}

%\begin{proposition}[{\rm see, e.g., \cite[Proposition 2.3]{Frommer-Simoncini}}]\label{p:f=p}
%Let functions $f$ and $g$ be analytic in a neighbourhood of the spectrum $\sigma(A)$ and
%\begin{equation*}
%f^{(j)}(\lambda_k)=g^{(j)}(\lambda_k),\qquad k=1,\dots,M;\;j=0,1,\dots,n_k-1.
%\end{equation*}
%Then
%\begin{equation*}
%f(A)=g(A).
%\end{equation*}
%\end{proposition}

\begin{proposition}[{\rm see, e.g., \cite[Proposition 2.3]{Frommer-Simoncini}}]\label{p:f->p}
Let $p$ be an interpolating polynomial of $f$ that corresponds to the points $\lambda_1,\dots,\lambda_M$ of the spectrum of the matrix $A$ counted according to their
multiplicities $n_1,\dots,n_M$.
Then
\begin{equation*}
p(A)=f(A).
\end{equation*}
\end{proposition}
%\begin{proof}
%The statement immediately follows from Proposition~\ref{p:f=p}.
%\end{proof}

\begin{remark}\label{r:n_k}
Proposition~\ref{p:f->p} remains valid if one assumes that $n_1$, \dots, $n_M$ are the maximal sizes of the corresponding Jordan blocks. This assumption decreases the degree $N-1$ of the interpolating polynomial.
\end{remark}

\section{Green's function}\label{s:Green}
In this Section, we recall the definition and some properties of Green's function.

Let $A$ be a complex $N\times N$-matrix. We consider the differential equation
\begin{equation}\label{e:x'=Ax+f}
x'(t)=Ax(t)+f(t),\qquad t\in\mathbb R.
\end{equation}
We are interested in \emph{bounded solutions problem}, i.e. seeking bounded solutions $x:\,\mathbb R\to\mathbb C^N$ under the assumption that the free term $f:\,\mathbb R\to\mathbb C^N$ is a bounded function. The bounded solutions problem has its origin in the work of Perron~\cite{Perron}. Its different modifications can be found in~\cite{BaskakovMS15:eng,Burd-Kolesov-Krasnoselskii69,Daletskii-Krein:eng,Henry81:eng,
KurbatovSMZ86:eng,Mukhamadiev81:eng,Pechkurov12:eng,Pokutnyi,Przeradzki}; see also references therein.

Suppose that $\sigma(A)$ does not intersect the imaginary axis. In this case the functions
\begin{align*}
\exp^+_t(\lambda)&=\begin{cases}
e^{\lambda t}, & \text{if $\Real\lambda<0$},\\
0, & \text{if $\Real\lambda>0$},\end{cases}\\
%\end{align*}
%\begin{align*}
\exp^-_t(\lambda)&=\begin{cases}
0, & \text{if $\Real\lambda<0$},\\
e^{\lambda t}, & \text{if $\Real\lambda>0$},\end{cases}\\
g_t(\lambda)&=\begin{cases}
-\exp^-_t(\lambda), & \text{if $t<0$},\\
\exp^+_t(\lambda), & \text{if $t>0$}
\end{cases}
\end{align*}
are analytic in the neighbourhood $\mathbb C\setminus i\mathbb R$ of the spectrum $\sigma(A)$. We set
\begin{equation}\label{e:Green's function}
\mathcal G(t)=g_t(A),\qquad t\neq0.
\end{equation}
The function $\mathcal G$ is called~\cite{Daletskii-Krein:eng} \emph{Green's function} of the bounded solutions problem for equation~\eqref{e:x'=Ax+f}.

%The following proposition is well known.
%
%\begin{proposition}\label{p:G-properties}
%Green's function possesses the properties{\rm:}
%\begin{enumerate}
% \item $P^+=\mathcal G(+0)$ and $P^-=\mathcal G(-0)$ are projectors, i.~e. they satisfy the identity $P^2=P$,
% \item $P^+-P^-=\mathbf1$,
% \item $\mathcal G(t_1)\mathcal G(t_2)=\mathcal G(t_1+t_2)$,
% \item $\mathcal G(t_1)\mathcal G(t_2)=0$ for $t_1t_2<0$.
% \item $\frac d{dt}\mathcal G(t)=A\mathcal G(t)$ for $t\neq0$.
%\end{enumerate}
%\end{proposition}
%\begin{proof}
%The statement follows from Proposition~\ref{p:func calc} and the identities
%\begin{gather*}
%g_{\pm0}^2(\lambda)=g_{\pm0}(\lambda),\\
%g_{+0}(\lambda)-g_{-0}(\lambda)=1,\\
%g_{t_1}(\lambda)g_{t_2}(\lambda)=g_{t_1+t_2}(\lambda)\text{ for }t_1t_2>0,\\
%g_{t_1}(\lambda)g_{t_2}(\lambda)=0\text{ for }t_1t_2<0,\\
%\frac{d}{dt}g_{t}(\lambda)=\lambda g_{t}(\lambda).\qed
%\end{gather*}
%\renewcommand\qed{}
%\end{proof}

The main property of Green's function is described in the following theorem.
\begin{theorem}[{\rm\cite[Theorem 4.1, p.~81]{Daletskii-Krein:eng}}]\label{t:Green}
Equation~\eqref{e:x'=Ax+f} has a unique bounded on $\mathbb R$ continuously differentiable solution $x$ for any bounded continuous
function $f$ if and only if the spectrum $\sigma(A)$ does not intersect the
imaginary axis.
This solution possesses the representation
\begin{equation*}
x(t)=\int_{-\infty}^\infty \mathcal G(t-s)f(s)\,ds,
\end{equation*}
where $\mathcal G$ is Green's function~\eqref{e:Green's function} of equation~\eqref{e:x'=Ax+f}.
\end{theorem}

Below we assume that $A$ is a fixed complex $N\times N$-matrix and its spectrum does not intersect the imaginary axis. We denote by $\mu_1,\dots,\mu_k$ the roots of the characteristic polynomial that lie in the open right half-plane $\Real\mu>0$ counted according to their multiplicities; and we denote by $\nu_1,\dots,\nu_m$ the roots of the characteristic polynomial that lie in the open left half-plane $\Real\nu<0$ counted according to their multiplicities. Thus, $k+m=N$.
We denote by $\gamma_-,\gamma_+>0$ real numbers such that
\begin{equation}\label{e:gamma pm}
\begin{split}
 \Real\mu_i&\ge\gamma_+\qquad\text{ for }1\le i\le k,\\
 \Real\nu_j&\le-\gamma_-\qquad\text{ for }1\le j\le m.
\end{split}
\end{equation}

\begin{proposition}[{\rm\cite{Kurbatov-Kurbatova17:Arxiv}}]\label{p:mu-nu:2}
Let an analytic function $f$ be identically zero in the open right half-plane $\Real\mu>0$ {\rm(}an example of such a function is the function $\exp^+_t${\rm)}. Then
\begin{equation*}
f[\mu_1,\dots,\mu_k;\nu_{1},\dots,\nu_m]=\tilde f[\nu_{1},\dots,\nu_m],
\end{equation*}
where
\begin{equation*}
\tilde f(z)=\frac{f(z)}{\prod_{i=1}^k(z-\mu_i)}.
\end{equation*}
\end{proposition}
\begin{proof}
Suppose that all multiplicities equal 1.
By Proposition~\ref{p:repr of Delta}, we have
\begin{align*}
f[\mu_1,\dots,\mu_k;\nu_{1},\dots,\nu_m]&=\sum_{q=1}^m\frac{f(\nu_q)}
{\prod\limits_{i=1}^k(\nu_q-\mu_i)\prod\limits_{\substack{j=1\\j\neq q}}^m(\nu_q-\nu_j)}
=\sum_{q=1}^m\frac{\frac{f(\nu_q)}
{\prod\limits_{i=1}^k(\nu_q-\mu_i)}}{\prod\limits_{\substack{j=1\\j\neq q}}^m(\nu_q-\nu_j)}\\
&=\tilde f[\nu_{1},\dots,\nu_m].
\end{align*}
From Proposition~\ref{p:f[] via Gamma} it easily follows that
divided differences continuously depend on their arguments. Hence, the case of multiple points of interpolation is obtained by a passage to the limit.
\end{proof}

\begin{theorem}[{\rm\cite{Kurbatov-Kurbatova17:Arxiv}}]\label{t:p_t}
Let us arrange the roots of the characteristic polynomial in the following order{\rm:}
\begin{equation}\label{e:order}
\mu_1,\dots,\mu_k;\,\nu_1,\dots,\nu_m.
\end{equation}
Then the Newton interpolating polynomial $p_t^+$ of the function $\exp^+_t$ takes the form
\begin{equation}\label{e:p+}
p_t^+(z)=(z-\mu_1)\dots(z-\mu_k)q_t^+(z),
\end{equation}
where
\begin{equation*}
q_t^+(z)=\widetilde{\exp^+_t}[\nu_{1}]+\dots+\widetilde{\exp^+_t}[\nu_{1},\dots,\nu_m](z-\nu_{1})\dots(z-\nu_{m-1})
\end{equation*}
is the interpolating polynomial of the function
\begin{equation*}
\widetilde{\exp^+_t}(z)=\frac{\exp^+_t(z)}{\prod_{i=1}^k(z-\mu_i)}
\end{equation*}
with respect to the points $\nu_1,\dots,\nu_m$.
The interpolating polynomial $p_t^-$ of the function $\exp^-_t$ can be represented in the form
\begin{equation*}%\label{e:p-}
p_t^-(z)=(z-\nu_1)\dots(z-\nu_m)q_t^-(z),
\end{equation*}
where
\begin{equation*}
q_t^-(z)=\widetilde{\exp^-_t}[\mu_{1}]+\dots+\widetilde{\exp^-_t}[\mu_{1},\dots,\mu_k](z-\mu_{1})\dots(z-\mu_{k-1})
\end{equation*}
is the interpolating polynomial of the function
\begin{equation*}
\widetilde{\exp^-_t}(z)=\frac{\exp^-_t(z)}{\prod_{j=1}^m(z-\nu_i)}
\end{equation*}
with respect to the points $\mu_1,\dots,\mu_k$.
\end{theorem}
\begin{proof}
We observe that $\exp^+_t(\mu_i)=0$, $i=1,\dots,k$. Therefore
\begin{equation*}
\exp^+_t[\mu_1]=\dots=\exp^+_t[\mu_1,\dots,\mu_k]=0.
\end{equation*}
Now from Proposition~\ref{p:Newton poly} it follows that
\begin{align*}
p_t^+(z)&=\exp^+_t[\mu_1,\dots,\mu_k;\nu_{1}](z-\mu_1)\dots(z-\mu_k)+\dots\\
&+\exp^+_t[\mu_1,\dots,\mu_k;\nu_{1},\dots,\nu_m](z-\mu_1)\dots(z-\mu_k)(z-\nu_{1})\dots(z-\nu_{m-1}).
\end{align*}
It remains to apply Proposition~\ref{p:mu-nu:2}.
\end{proof}

\section{The estimate}\label{s:estimate}
In this Section we prove an estimate of Green's function. As a potential application of this estimate, we note that knowing an estimate of the function $t\mapsto\Vert\mathcal G(t)\Vert$ is an important information in the freezing method for equations with slowly varying coefficients~\cite[\S~10.2]{Bylov-Vinograd:rus1},~\cite[\S~7.4]{Henry81:eng},~\cite[ch.~10, \S~3]{Levitan-Zhikov82:eng},~\cite{Baskakov93:eng,Behncke-Hinton-Remling,Kuznetsova85:eng,
Kuznetsova90:eng,Potzsche,Robinson,Xiao}. See also references therein.

\begin{lemma}\label{l:der est}
Let $\gamma^-,\gamma^+>0$, $\Real z\le-\gamma^-$, and $\Real\mu_j\ge\gamma^+$ for $j=1,\dots,k$. Then for $k\ge1$ we have
\begin{equation}\label{e:est1}
\biggl|\frac{d^l}{dz^l}\frac{e^{zt}}{\prod_{j=1}^k(z-\mu_j)}\biggr|\le
e^{-\gamma^- t}\sum_{i=0}^lt^{l-i}\binom{l}{i}\frac{(k+i-1)!}{(k-1)!}\frac1{\gamma^{k+i}},\qquad t>0,
\end{equation}
where $\gamma=\gamma^-+\gamma^+$.
But for $k=0$
\begin{equation}\label{e:est2}
\Bigl|\frac{d^l}{dz^l}e^{zt}\Bigr|\le e^{-\gamma^- t}\,t^{l},\qquad t>0.
\end{equation}
\end{lemma}

\begin{remark}\label{r:1}
Formula~\eqref{e:est2} becomes a special case of~\eqref{e:est1} if one sets
$\frac{(-1)!}{(-1)!}=1$ and $\frac{(i-1)!}{(-1)!}=0$ for $i=1,2,\dots$.
\end{remark}
\begin{proof}
By the general Leibniz product differentiation rule~\cite[ch. 1, \S~3, Proposition 2]{Bourbaki-FRV:eng} we have the identity
\begin{equation*}
\biggl[\frac{e^{zt}}{\prod_{j=1}^k(z-\mu_j)}\biggr]^{(l)}=
e^{zt}\sum_{i=0}^l\binom{l}{i}t^{m-i}\biggl[\frac{1}{\prod_{j=1}^k(z-\mu_j)}\biggr]^{(i)}.
\end{equation*}
In order to complete the proof, it is enough to show that
\begin{equation*}
\biggl|\biggl[\frac{1}{\prod_{j=1}^k(z-\mu_j)}\biggr]^{(i)}\biggr|\le
\frac{(k+i-1)!}{(k-1)!}\frac1{\gamma^{k+i}}.
\end{equation*}
We recall that among the numbers $\mu_j$ there may be repeating ones.

We note that by the product differentiation rule~\cite[ch. 1, \S~1, Proposition 3]{Bourbaki-FRV:eng}, the derivative
$\Bigl[\frac{1}{\prod_{j=1}^k(z-\mu_j)}\Bigr]^{'}$ is the sum of $k$ summands of the form $\frac{-1}{\prod_{j=1}^{k+1}(z-\mu_j^{(1)})}$, where $\mu_j^{(1)}$ are the old numbers $\mu_j$, but one of them is repeated twice. In the course of the next differentiation, each term $\frac{1}{\prod_{j=1}^{k+1}(z-\mu_j^{(1)})}$ turns into $k+1$ terms of the form $\frac{-1}{\prod_{j=1}^{k+2}(z-\mu_j^{(2)})}$, and the entire first derivative is transformed into $k(k+1)$ terms of the form $\frac{1}{\prod_{j=1}^{k+2}(z-\mu_j^{(2)})}$, where $\mu_j^{(2)}$ are some numbers satisfying the condition $\Real\mu_j^{(2)}\ge\gamma^+$. The third derivative $\Bigl[\frac{1}{\prod_{j=1}^k(z-\mu_j)}\Bigr]^{(3)}$ consists of $k(k+1)(k+2)$ summands of the form $\frac{-1}{\prod_{j=1}^{k+3}(z-\mu_j^{(3)})}$. And so on.

Each term of the $i$-th derivative is less in absolute value than or equal to $\frac1{\gamma^{k+i}}$, and the total number of terms is $k(k+1)(k+2)(k+i-1)=\frac{(k+i-1)!}{(k-1)!}$.
\end{proof}

Let us fix a norm in $\mathbb C^N$. We define the norm of an $N\times N$-matrix $A$ as the norm of the linear operator acting in $\mathbb C^N$ induced by $A$.

\begin{theorem}\label{t:VinGr}
Let assumption~\eqref{e:gamma pm} be satisfied. We set $\gamma=\gamma^-+\gamma^-$.
Then Green's function satisfies the estimates
\begin{align}
\Vert\mathcal G(t)\Vert&\le e^{-\gamma^-t}\sum_{j=0}^{m-1}\frac{t^{j-i}}{(j-i)!}
\sum_{i=0}^j\binom{k+i-1}{k-1}
\frac{(2\Vert A\Vert)^{k+j}}{\gamma^{k+i}},& t&>0,\label{e:estG1}\\
\Vert\mathcal G(t)\Vert&\le e^{\gamma^+t}\sum_{j=0}^{k-1}\frac{t^{j-i}}{(j-i)!}
\sum_{i=0}^j\binom{m+i-1}{m-1}
\frac{(2\Vert A\Vert)^{m+j}}{\gamma^{m+i}},& t&<0.\label{e:estG2}
\end{align}
\end{theorem}
\begin{proof}
We consider the case $t>0$. We represent the Newton interpolating polynomial $p_t^+$ of the function $\exp^+_t$ in the form~\eqref{e:p+}.
By Propositions~\ref{p:func calc} and~\ref{p:f->p} we have
\begin{equation*}
\mathcal G(t)=\exp^+_t(A)=p_t^+(A)=(A-\mu_1\mathbf1)\dots(A-\mu_k\mathbf1)q_t^+(A),\qquad t>0,
\end{equation*}
where
\begin{equation}\label{e:q_t+}
q_t^+(A)=\widetilde{\exp^+_t}[\nu_{1}]\mathbf1+\dots+\widetilde{\exp^+_t}[\nu_{1},\dots,\nu_m](A-\nu_{1}\mathbf1)\dots(A-\nu_{m-1}\mathbf1).
\end{equation}
Clearly, $\Vert A-\mu_i\mathbf1\Vert\le2\Vert A\Vert$.
Therefore
\begin{equation}\label{e:G-est}
\Vert\mathcal G(t)\Vert=\Vert p_t(A)\Vert\le(2\Vert A\Vert)^k\Vert q_t^+(A)\Vert.
\end{equation}

From representation~\eqref{e:q_t+}, Proposition~\ref{p:Gelfond(49)}, Theorem~\ref{t:p_t}, and Lemma~\ref{l:der est} we have
\begin{align*}
\Vert q_t^+(A)\Vert&\le\sum_{j=0}^{m-1}\Bigl|\widetilde{\exp^+_t}[\nu_{1},\dots,\nu_{j+1}]\Bigr|
(2\Vert A\Vert)^{j}\\
&\le\sum_{j=0}^{m-1}\frac1{j!}\max_{\lambda\in\co\{\nu_{1},\dots,\nu_{j+1}\}}
\Bigl|\widetilde{\exp^+_t}^{(j)}(\lambda)\Bigr|(2\Vert A\Vert)^{j}\\
&\le e^{-\gamma^-t}\sum_{j=0}^{m-1}\frac1{j!}
\sum_{i=0}^jt^{j-i}\binom{j}{i}\frac{(k+i-1)!}{(k-1)!}
\frac{(2\Vert A\Vert)^{j}}{\gamma^{k+i}}\\
&\le e^{-\gamma^-t}\sum_{j=0}^{m-1}
\sum_{i=0}^jt^{j-i}\frac{1}{i!(j-i)!}\frac{(k+i-1)!}{(k-1)!}\frac{(2\Vert A\Vert)^{j}}{\gamma^{k+i}}\\
&\le e^{-\gamma^-t}\sum_{j=0}^{m-1}
\sum_{i=0}^j\frac{t^{j-i}}{(j-i)!}\binom{k+i-1}{k-1}
\frac{(2\Vert A\Vert)^{j}}{\gamma^{k+i}}.
%\\
%&\le e^{-\gamma^-t}\sum_{i=0}^{m-1}\sum_{j=i}^{m-1}
%\frac{t^{j-i}}{(j-i)!}\binom{k+i-1}{k-1}
%\frac{(2\Vert A\Vert)^{j}}{\gamma^{k+i}}\\
%&\le e^{-\gamma^-t}\sum_{i=0}^{m-1}\binom{k+i-1}{k-1}\sum_{q=0}^{m-i-1}
%\frac{t^{q}}{q!}\frac{(2\Vert A\Vert)^{q+i}}{\gamma^{k+i}}\\
\end{align*}

%\begin{equation*}
%\bigl|f[\mu_{1},\dots,\mu_{N}]\bigr|\le\frac1{(N-1)!}\max_{\lambda\in\co\{\mu_{1},\dots,\mu_{N}\}}|f^{(N-1)}(\lambda)|.
%\end{equation*}
%\begin{equation*}
%\biggl|\biggl[\frac{e^{zt}}{\prod_{j=1}^k(z-\mu_j)}\biggr]^{(j)}\biggr|\le
%e^{-\gamma^- t}\sum_{i=0}^jt^{j-i}\binom{j}{i}\frac{(k+i-1)!}{(k-1)!}\frac1{\gamma^{k+i}},
%\end{equation*}
%\begin{equation*}
%\biggl|\biggl[\frac{e^{zt}}{\prod_{j=1}^k(z-\mu_j)}\biggr]^{(m)}\biggr|\le
%e^{-\gamma^- t}\sum_{i=0}^mt^{m-i}\binom{m}{i}\frac{(k+i-1)!}{(k-1)!}\frac1{\gamma^{k+i}},
%\end{equation*}

Taking~\eqref{e:G-est} into account we arrive at
\begin{align*}
\Vert\mathcal G(t)\Vert&\le e^{-\gamma^-t}\sum_{j=0}^{m-1}
\sum_{i=0}^j\frac{t^{j-i}}{(j-i)!}\binom{k+i-1}{k-1}
\frac{(2\Vert A\Vert)^{k+j}}{\gamma^{k+i}},& t&>0.
\end{align*}

Formula~\eqref{e:estG2} is proved in a similar way.
\end{proof}

\begin{example}\label{ex:1}
For $N=6$ and $k=3$ estimate~\eqref{e:estG1} has the form
\begin{align*}
\Vert\mathcal G(t)\Vert&\le e^{-\gamma^-t}\Bigl(\frac{6 \Vert A\Vert^5}{\gamma ^5}+\frac{6 \Vert A\Vert^5 t}{\gamma ^4}+\frac{2 \Vert A\Vert^5 t^2}{\gamma ^3}+\frac{3 \Vert A\Vert^4}{\gamma ^4}+\frac{2 \Vert A\Vert^4
   t}{\gamma ^3}+\frac{\Vert A\Vert^3}{\gamma ^3}\Bigr),\qquad t>0.
\end{align*}
\end{example}

\begin{corollary}[{\rm\cite[p.~131, Lemma 10.2.1]{Bylov-Vinograd:rus1}, \cite[p.~68, formula (13)]{Gelfand-Shilov-GF3:eng}}]\label{p:Bylov-Vinograd}
Let the eigenvalues of the matrix $A$ lie in the half-plane $\Real\lambda<-\gamma^-$, where $\gamma^->0$. Then
\begin{equation*}
\Vert e^{At}\Vert\le e^{-\gamma^- t}\sum_{j=0}^{N-1}\frac{(2t\Vert A\Vert)^j}{j!},\qquad t>0.
\end{equation*}
\end{corollary}
\begin{proof}
The proof is similar to that of Theorem~\ref{t:VinGr}.

The proof can also be obtained as a special case of~\eqref{e:estG1} if we take into account Remark~\ref{r:1}.
\end{proof}

%\bibliographystyle{plain} % gost2008s, abbrv, acm, alpha, apalike, ieeetr, plain, siam, unsrt, gost705s, gost2003s, gost2008ns, gost2008ls
%\bibliography{E:/Document/BIB/Barycentric,E:/Document/BIB/Gil,E:/Document/bib/NM,E:/Document/bib/Bourbaki,E:/Document/bib/FA,E:/Document/bib/DE,E:/Document/bib/Kurbatov}

\end{document}